\documentclass[12pt]{amsart}

\usepackage{tikz,amsthm,amsmath,amstext,amssymb,amscd,epsfig,euscript, mathrsfs, dsfont,pspicture,multicol,graphpap,graphics,graphicx,times,enumerate,subfig,sidecap,wrapfig,color,pict2e}

\usepackage[colorlinks,citecolor=red,pagebackref,hypertexnames=false]{hyperref}

\usepackage{enumerate}

 \numberwithin{equation}{section}

\def\bC{{\mathbb{C}}}

\def\bR{{\mathbb{R}}}
\def\R{{\mathbb{R}}}

\def\cC{{\mathscr{C}}}

\def\cH{{\mathscr{H}}}

\def\ve{\varepsilon}

\renewcommand{\d}{{\partial}}
\def\lec{\lesssim}
\def\gec{\gtrsim}

\DeclareMathOperator{\diam}{diam}
\def\dist{\mathop\mathrm{dist}} 						

\newcommand{\ps}[1]{\left( #1 \right)}

\newcommand{\av}[1]{\left| #1 \right|}

\newcommand{\cnj}[1]{\overline{#1}}

\def\Xint#1{\mathchoice
{\XXint\displaystyle\textstyle{#1}}%
{\XXint\textstyle\scriptstyle{#1}}%
{\XXint\scriptstyle\scriptscriptstyle{#1}}%
{\XXint\scriptscriptstyle\scriptscriptstyle{#1}}%
\!\int}
\def\XXint#1#2#3{{\setbox0=\hbox{$#1{#2#3}{\int}$ }
\vcenter{\hbox{$#2#3$ }}\kern-.58\wd0}}

\def\avint{\Xint-}
\def\grad{\nabla}

\theoremstyle{plain}
\newtheorem{theorem}{Theorem}

\newtheorem{corollary}[theorem]{Corollary}

\newtheorem{lemma}[theorem]{Lemma}

\theoremstyle{definition}

\newtheorem{definition}[theorem]{Definition}
\newtheorem{remark}[theorem]{Remark}

\numberwithin{equation}{section}
\numberwithin{theorem}{section}

\newtheorem{main}{Theorem}

  \DeclareFontFamily{U}{mathb}{\hyphenchar\font45} 
\DeclareFontShape{U}{mathb}{m}{n}{
      <5> <6> <7> <8> <9> <10> gen * mathb
      <10.95> mathb10 <12> <14.4> <17.28> <20.74> <24.88> mathb12
      }{}
\DeclareSymbolFont{mathb}{U}{mathb}{m}{n}

\DeclareMathSymbol{\toitself}{3}{mathb}{"FD}  

\begin{document}

\title[Accessible parts of the boundary]{Accessible parts of the boundary for domains with lower content regular complements}

\author[Azzam]{Jonas Azzam}
\address{Jonas Azzam\\
School of Mathematics \\ University of Edinburgh \\ JCMB, Kings Buildings \\
Mayfield Road, Edinburgh,
EH9 3JZ, Scotland.}
\email{j.azzam "at" ed.ac.uk}


\begin{abstract}
We show that if $0<t<s\leq n-1$, $\Omega\subseteq \R^{n}$ with lower $s$-content regular complement, and $z\in \Omega$, there is a chord-arc domain $\Omega_{z}\subseteq \Omega $ with center $z$ so that $\cH^{t}_{\infty}(\d\Omega_{z}\cap \d\Omega)\gec_{t} \dist(z,\Omega^{c})^{t}$. This was originally shown by Koskela, Nandi, and Nicolau with John domains in place of chord-arc domains when $n=2$, $s=1$, and $\Omega$ is a simply connected planar domain. 

Domains satisfying the conclusion of this result support $(p,\beta)$-Hardy inequalities for $\beta<p-n+t$ by a result of Koskela and Lehrb\"{a}ck; Lehrb\"{a}ck also showed that $s$-content regularity of the complement for some $s>n-p+\beta$ was necessary. Thus, the combination of these results characterizes when a domain supports a pointwise $(p,\beta)$-Hardy inequality for $\beta<p-1$ in terms of lower content regularity.


\end{abstract}
\maketitle
\tableofcontents

\section{Introduction}

In this note we study how accessible the boundary of a connected domain $\Omega\subseteq \bR^{n}$ is under certain nondegeneracy conditions on the boundary. By virtue of being connected, all points in the boundary are trivially accessible by a curve, but in some applications it is more important to have some non-tangential accessibility on a non-trivial portion of the boundary.

For a domain $\Omega$, $x\in \Omega$, and $c>0$, we say $\Omega$ is {\it $c$-John with center $x\in \Omega$} if every $y\in \cnj{\Omega}$ is connected to $x$ by a curve $\gamma$ so that
\begin{equation}
\label{e:john}
c\cdot  \ell(y,z)\leq \delta_{\Omega}(z):= \dist(z,\Omega^{c})\mbox{ for all $z\in \gamma$}\end{equation}
where $\ell(y,z)$ denotes the length of the subarc of $\gamma$ from $y$ to $z$. In this way, every point $y$ in the domain is non-tangentially accessible from $x$, that is, there is a curve about which the domain does not pinch as it approaches $y$.  We will let $v_{x}(c)$ denote the {\it $c$-visual boundary}, that is, the set of $z\in \d\Omega$ for which there is a curve $\gamma$ satisfying \eqref{e:john}.
 
Of course, most domains are not John and could pinch at many points in the boundary. However, if $\d\Omega$ is infinite, one can see that $v_{x}(c)$ should be infinite as well. It's natural to ask then how big the visual boundary can be.
%
%
%
%
%

Our main result states that, if the complement has large $s$-dimensional content uniformly with $s\leq n-1$, then the visual boundary also has large content with respect to any dimension less than $s$. In fact, we show that for any $t<s$, there is even a chord-arc subdomain intersecting a large $t$-dimensional portion of the boundary.

%

\begin{main}
\label{t:main}
Let $0<s\leq  n-1$, and suppose $\Omega\subseteq \mathbb{R}^{n}$ has lower $s$-content regular complement, meaning there is $c_0>0$ so that
\begin{equation}
\label{e:lcr}
\cH^{s}_{\infty}(B(x,r)\backslash \Omega)\geq c_0r^{s} \mbox{ for all }x\in \d\Omega, \;\; 0<r<\diam \d\Omega.
\end{equation}
Then for every $0<t<s$, $\Omega$ has {\it big $t$-pieces of chord-arc subdomains} (or $BPCAS(t)$), meaning there is $C=C(s,t,n,c_{0})>0$ so that for all $x\in \Omega$ with $\dist(x,\d\Omega)< \diam \d\Omega$, there is a $C$-chord-arc domain $\Omega_{x}$ with center $x$ so that 
\[
\cH^{t}_{\infty}(\d\Omega\cap \d\Omega_{x})\geq C^{-1} \delta_{\Omega}(x)^{t}.
\]
In particular, there is $c=c(s,t,n,c_{0})>0$ so that 
\begin{equation}
\label{e:visual}
\cH^{t}_{\infty}(v_{x}(c))\geq C^{-1} \delta_{\Omega}(x)^{t}
\end{equation}
\end{main}

We will define chord-arc domains later (see Definition \ref{d:definition}), but in particular, when bounded, they are John domains.

for $A\subseteq \bR^{n}$, we define its {\it $s$-dimensional Hausdorff content} as
\[
\cH^{s}_{\infty}(A):=\inf\left\{\sum (\diam A_{i})^{s}: A\subseteq \bigcup A_{i}\right\}.
\]

Recently, Koskela, Nandi and Nicolau in \cite{KNN18} showed \eqref{e:visual} holds for simply connected planar domains $\Omega$, when $n=2$ and $t<s=1$ using techniques from complex analysis. 

The conclusion fails for $t=n-1$, even when $\Omega$ has some nice geometry. Indeed, suppose $\Omega\subseteq \bR^{n}$ had uniformly rectifiable boundary and the interior corkscrew condition, then \eqref{e:visual} with $t=n-1$ is exactly the {\it weak local John condition} introduced by Hofmann and Martell. They show that this implies the weak-$A_{\infty}$ property for harmonic measure \cite{HM18}, and in particular, that harmonic measure is absolutely continuous with respect to surface measure, although there are examples of such domains where this isn't the case \cite{BJ90}. 

The theorem also does not hold for $s>n-1$. The counter example is essentially the same example made by Koskela and Lehrb\"{a}ck in \cite[Example 7.3]{KL09}: Let $A$ be (see Figure \ref{f:antenna}) the self-similar fractal in $\bC$ determined the following similarities:
\[f_{1}(z)=\frac{z}{2},\;\; f_{2}(z)=\frac{z+1}{2}, \;\; f_{3}(z)=i\alpha z+\frac{1}{2},f_{4}(z)=-i\alpha z+\frac{1}{2}+i\alpha\]
where $\alpha\in (0,\frac{1}{2})$ is some fixed number (for a reference on self-similar fractals, see \cite[Section 8.3]{Falconer}. Let $\Omega=\bC\backslash A$, then $A$ satisfies $\cH^{s}(B(x,r)\cap \Omega^{c})\sim r^{s}$ for some $s>1$ and all $0<r<1$ yet, by picking $x$ closer and closer to the flatter side of $A$, a John domain with center $x$ intersecting $A$ in a $s$-dimensional portion of the boundary must wrap around to the other side of the antenna, hence the John constant will blow up as $x$ approaches the flat part. 

\begin{figure}[!ht]
\includegraphics[width=300pt]{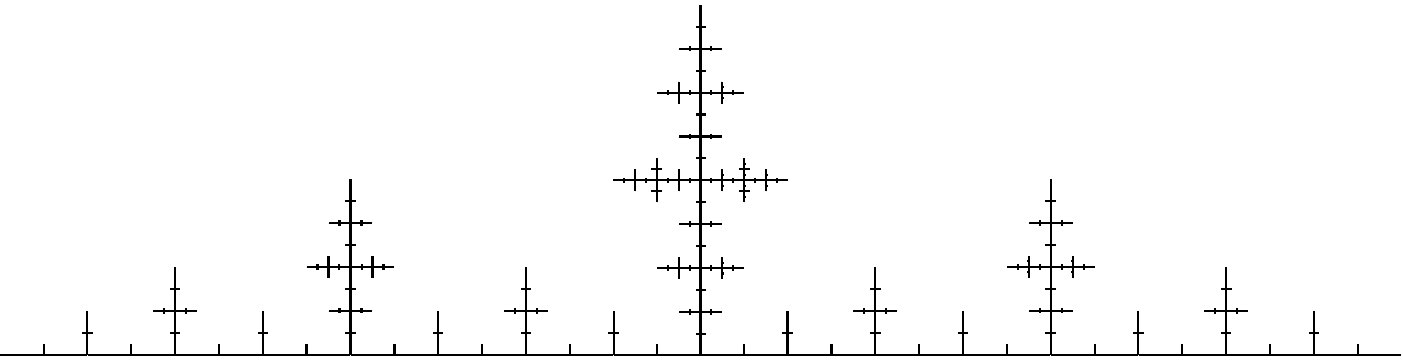}
\caption{The antenna set.}
\label{f:antenna}
\end{figure}

We don't know about the case $s=t<n-1$ and whether it should hold.
%
%
%
%
%
%

The existence of  accessible portions of the boundary has been investigated previously due to its connections to Hardy-type inequalities. 

\begin{definition}
A domain $\Omega$ satisfies the {\it $(p,\beta)$-Hardy inequality} if
\[
\int_{\Omega}|u(x)|^{p}\delta_{\Omega}(x)^{\beta-p}dx \lec \int_{\Omega}|\grad u(x)|^{p}\delta_{\Omega}(x)^{\beta}dx \;\;\;\mbox{ for all } u\in C_{0}^{\infty}(\Omega).
\]
We also say $\Omega$ satisfies a {\it pointwise $(p,\beta)$-Hardy inequality}  if there is $q\in (1,p)$ such that for all $x\in \Omega$ and $u\in C_{0}^{\infty}(\Omega)$,
\[
|u(x)|\lec \delta_{\Omega}(x)^{1-\frac{\beta}{p}} 
\ps{ \sup_{r<2\delta_{\Omega}(x)} \avint_{B(x,r)} |\grad u(y)|^{q}\delta_{\Omega}(y)^{q \frac{\beta}{p}}dy}^{\frac{1}{q}}.
\]
\end{definition}

Koskela and Lehrb\"{a}ck showed the following in \cite[Proposition 5.1]{KL09}:

%

\begin{theorem}
\label{t:KL09}
Let $\Omega\subseteq \R^{n}$ and suppose \eqref{e:visual} holds for some $0\leq t\leq n$. Then $\Omega$ satisfies a pointwise $(p,\beta)$-Hardy inequality for all $\beta<p-n+t$.
\end{theorem}

From this, they could also show that the $(p,\beta)$-Hardy inequality holds for all $\beta<p-n+t$ as well \cite[Theorem 1.4]{KL09}, however Lehrb\"{a}ck later showed in \cite{Leh14} that \eqref{e:visual} is not necessary to prove this. He also generalizes Theorem \ref{t:KL09} to metric spaces with a suitable substitute for \eqref{e:visual}. 

What is not known is whether having lower content regular complements alone implies pointwise Hardy inequalities without any assumptions on the visual boundary (see the discussion at the top of \cite[p. 1707]{Leh14}). In \cite{Leh14}, Lehrb\"{a}ck shows that they do hold if $\beta\leq 0$  and $\beta<p-n+t$, and so the gap in our knowledge is whether they hold when $0<\beta<p-n+t$. In \cite{Leh09},  however, he shows lower content regularity is necessary:

\begin{theorem}
If $\Omega\subseteq \R^{n}$ admits the pointwise $(p,\beta)$-Hardy inequality, then there is $s > n-p+\beta$ so that $\Omega$ has lower $s$-content regular complement. 
\end{theorem}

As a corollary of Theorem \ref{t:main} and Theorem \ref{t:KL09},  we get that the lower content regularity is also sufficient, and thus combined with the previous theorem, we get the following characterization.

\begin{corollary}
Let $\Omega\subseteq \R^{n}$, $\beta\in \bR$ and $1<p<\infty$ and $\beta<p-1$. Then $\Omega$ satisfies the $(p,\beta)$-pointwise Hardy inequality if and only if there is $s>n-p+\beta$ for which $\Omega$ has lower $s$-content regular complement.  
\end{corollary}

Indeed, if $\Omega$ has lower $s$-content regular complement, then Theorem \ref{t:main} says \eqref{e:visual} holds for any $t<s$. Theorem \ref{t:KL09} implies it satisfies the pointwise $(p,\beta)$-Hardy inequality for all $\beta<p-n+t$, and hence (letting $t\uparrow s$) for all $\beta<p-n+s$. 

Note that if $\Omega$ is $s$-content regular for some $s>n-1$, then it is also $(n-1)$-content regular, so the above corollary yields the $(p,\beta)$-Hardy inequality for all $\beta<p-1$. 
In Theorem 1.3 in \cite{KL09}, the authors also show that for every $1<s <2$, there is a simply connected domain $\Omega\subseteq \bC$ with lower $s$-content regular complement yet the $(p,p-1)$-Hardy inequality fails. Thus, for lower $(n-1)$-content regular domains, the bound $\beta<p-1$ is tight. 

Condition \eqref{e:visual} implies other Hardy-type inequalities. For example, in \cite{ILTV14}, Ihnatsyeva, Lehrb\"{a}ck, Tuominen, and V\"{a}h\"{a}kangas show that \eqref{e:visual} implies certain fractional Hardy inequalities. \\

%
%
%
%

The structure of the proof of Theorem \ref{t:main} goes roughly as follows. The aim is to construct a tree of tentacles emanating from $x$ whose endpoints are a large subset of the boundary, and then we take an appropriate neighborhood of this tree. Given a point $x\in \Omega$, the boundary has large Hausdorff content near $x$. This means that, for a large set of directions, the orthogonal projection of the boundary has large $t$-content for some $t<s$ of our choosing. We use this to construct a tree of points $\{x_{\alpha}\}$ where $\alpha$ is a multi-index as follows: let $\ve>0$ be small, set  $x_{\emptyset}=x$ and assume without loss of generality that $\delta_{\Omega}(x)=1$. Given a point $x_{\alpha}$ so that $\delta_{\Omega}(x)=\ve^{|\alpha|}$, if $\xi_{\alpha}\in \d\Omega$ is closest to $x_{\alpha}$, find a plane $V_{\alpha}$ passing through $x_{\alpha}$ in which projection of $B(\xi_{\alpha},\ve^{|\alpha|}/4)\cap \d\Omega$ is large in a small ball around $x_{\alpha}$. We can then pick a maximally $M\ve^{|\alpha|+1}$-separated collection of points $\{y_{\alpha i}\}_{i}$ (where $M$ is some large number) in the projection of $B(\xi_{\alpha},\ve^{|\alpha|}/4)\cap \d\Omega$. For each $y_{\alpha i}$ we move up (perpendicular to $P_{\alpha}$) toward the boundary until we find points $x_{\alpha i}$ with distance $\ve^{|\alpha|+1}$ from $\d\Omega$. We then repeat the process on these points, and so on so forth. The union of $\frac{1}{4}B_{\alpha}\cup \bigcup_{i} [y_{\alpha i},x_{\alpha i}]$ over all $\alpha$ will be a connected subset of $\Omega$ and, choosing parameters correctly, will have closure intersecting a part of the boundary with large $t$-content. We fatten this set up by taking the union over dilated Whitney cubes intersecting this set. 
%
Then we show that this resulting domain is in fact chord-arc, the proof of which follows roughly the same procedure that has been done in several papers about harmonic measure, see for example \cite{HM14}. \\

We'd like to thank Riikka Korte, Pekka Koskela and Juha Lehrb\"{a}ck for answering our questions and commenting on the manuscript.

\section{Notation}

We say $a\lec b$ if there is a constant $C$ so that $a\leq Cb$, and $a\lec_{t}b $ if $C$ depends on the parameter $t$. We also write $a\sim b$ if $a\lec b\lec a$ and define $a\sim_{t} b$ similarly. We will omit the dependence on $n$ throughout the paper.

We will let $B(x,r)$ denote the open ball centered at $x$ of radius $r$. If $B=B(x,r)$ and $c>0$, we let $cB=B(x,cr)$. Similarly, if a cube $Q\subseteq \bR^{n}$ with sides parallel to the coordinate axes has center $x$, we denote its side length by $\ell(Q)$ and write $cQ$ for the cube of same center and sides still parallel to the coordinate axes but with side length equal to $c\ell(Q)$.

For $A$ and $B$ sets, and $x\in \bR^{n}$, we define
\[
\dist(x,A)=\inf_{y\in A}|x-y|,\;\; \dist(A,B)=\inf_{x\in B}\dist(x,A)  
\]
and 
\[
\diam A=\sup\{|x-y|:x,y\in A\}.
\]


\section{A Lemma about the Hausdorff content of projections}

We will need to know that if a set has large Hausdorff content, then so does its orthogonal projection in most directions, at least with respect to a smaller dimension. Its proof follows the computations in Chapter 9 of \cite{Mattila}, the only difference being we take more care in order to make them quantitative. We recall that a measure $\mu$ is a $t$-Frostmann measure if 
\[
\mu(B(x,r))\leq r^{t} \mbox{ for all }x\in \R^{n} \mbox{ and }r>0\]
and it is not hard to show that for a $t$-Frostmann measure
\begin{equation}
\label{e:mu<H}
\mu(E)\lec  \cH^{t}_{\infty}(E) \mbox{ for all }E\subseteq \R^{n}.
\end{equation}

Also recall that $G(n,m)$ denotes the {\it Grassmannian} of $m$-dimensional planes in $\R^{n }$ and $\gamma_{n,m}$ is the Grassmannian measure on $G(n,m)$. For a reference, see \cite[Chapter 3]{Mattila}. 

\begin{lemma}
\label{l:mattila}
Let $0<m<n$ be integers, $0<t<s\leq m$ and let $E$ be a compact set. Then for any $V_{0}\in G(n,m)$ and $\delta>0$ there is $V\in G(n,m)$ so that $d(V_{0},V)<\delta$ and, if $P_{V}$ is the orthogonal projection into $V$,
\begin{equation}
\label{e:mattila}
\cH^{t}_{\infty}(P_{V}(E))\gec_{\delta,n,t,s} (\diam E)^{t-s} \cH^{s}_{\infty}(E)
\end{equation}
\end{lemma}

\begin{proof}
Let $\mu$ be a $s$-Frostmann measure on $E$ so that $\mu(B(x,r))\leq r^{s}$ for all $x\in \bR^{n}$ and $r>0$ and so that $\mu(E)\sim_{n} \cH^{s}_{\infty}(E)$ (see \cite[Theorem 8.8]{Mattila}). Let $A=\{V:d(V,V_{0})<\delta\}$. By \cite[Corollary 3.12]{Mattila}, 
\[
\int_{G(n,m)} |P_{V}(x)|^{-t}d\gamma_{n,m}(V)\leq \ps{1+\frac{2^{n}t}{\alpha(n)(m-t)}}|x|^{-t} =: \frac{|x|^{-t}}{c},\]

%

Let $F:=P_{V}(E)$. Then

\begin{align*}
I_{t}(\mu)
& :=\int_{E}\int_{E} |x-y|^{-t}d\mu(x)d\mu(y)
\\
& \geq  c \int_{A}\int_{E}\int_{E} |P_{V}(x-y)|^{-t} d\mu(x)d\mu(y)d\gamma_{n,m}(V)\\
& =  c\gamma_{n,m}(A) \avint_{A}\int_{F} \underbrace{\int_{F} |x-y|^{-t}dP_{V}[\mu](x)}_{=:E(y)} dP_{V}[\mu](y) d\gamma_{n,m}(V).
\end{align*}
Hence, there is $V\in A$ so that 
\begin{align*}
C & :=\frac{I_{t}(\mu)}{c\gamma_{n,m}(A)} \geq \int_{F}  E(y) dP_{V}[\mu](y)\\
& =\int_{0}^{\infty} P_{V}[\mu](\{ y\in F:E(y)>\lambda\})d\lambda.
\end{align*}
This implies there must be $\lambda \in [0,2C/P_{V}[\mu](F)]$ so that 
\[
P_{V}[\mu](\{ y\in F:E(y)>\lambda\})\leq P_{V}[\mu](F)/2.\]
 Hence, if $S=\{ y\in F:E(y)\leq \lambda\}$, we have 
 \begin{equation}
 \label{e:Sbig}
 P_{V}[\mu](S)\geq P_{V}[\mu](F)/2.
\end{equation}

Let $\nu = P_{V}[\mu]|_{S}$. Then for $y\in S$ and $r>0$, 
\begin{align*}
\nu(B(y,r))r^{-t} 
& \leq \int_{B(y,r) \cap F}|x-y|^{-t}dP_{V}[\mu](x)
=E(y)\leq \lambda \\
&  \leq \frac{2C}{P_{V}[\mu](F)}.
\end{align*}
Hence, $ \frac{P_{V}[\mu](F)}{2C}\nu$ is a $t$-Frostmann measure on $F$. Thus, since $\nu(F)=P_{V}[\mu](S)\stackrel{\eqref{e:Sbig}}{\geq}P_{V}[\mu](F)/2$,
\begin{equation}
\label{e:CH}
C\cH^{t}_{\infty}(F) \stackrel{\eqref{e:mu<H}}{\gec}   P_{V}[\mu](F) \nu(F)
\stackrel{\eqref{e:Sbig}}{\geq}
\frac{P_{V}[\mu](F)^2}{2} = \frac{\mu(E)^2}{2}\sim_{n} \cH^{s}_{\infty}(E)^2.
\end{equation}

Note that 
\begin{align*}
\int |x-y|^{-t}d\mu(y)
& =\int_{0}^{\infty} \mu(\{y:|x-y|^{-t}>a\})da\\
& =\int_{0}^{\infty} \mu(\{y:|x-y|<a^{-1/t}\})da\\
& = \int_{0}^{\infty} \mu(B(x,a^{-1/t}))da\\
& \leq \int_{0}^{(2\diam E)^{-t}}\mu(B(x,2\diam E))+\int_{(2\diam E)^{-t}}^{\infty}  a^{-s/t}da\\
& \leq \frac{\mu(E)}{(2\diam E)^{t}} -\frac{((2\diam E)^{-t})^{-s/t+1}}{-s/t+1}\\
& \leq \frac{\mu(E)}{(\diam E)^{t}} + \frac{t (2\diam E)^{s-t}}{s-t}\\
& \lec \frac{\cH^{s}_{\infty}(E)}{(\diam E)^{t}} +\frac{t 2^{s-t}}{s-t} (\diam E)^{s-t}.
\end{align*}
Hence, since $\mu(\bR^{n})\sim_{n} \cH^{s}_{\infty}(E)$, we get 
\[
C\sim_{\delta,n} I_{t}(\mu)
\lec\cH^{s}_{\infty}(E)\ps{ \frac{\cH^{s}_{\infty}(E)}{(\diam E)^{t}} +\frac{t 2^{s-t}}{s-t}(\diam E)^{s-t}}\]

Recalling that $F=P_{V})(E)$, we have 
\begin{align*}
\cH^{t}_{\infty}(P_{V}(E)) \stackrel{\eqref{e:CH}}{\gec}_{n} 
& \gec \frac{\cH^{s}_{\infty}(E)^{2}}{C}\gec_{\delta,n} 
\frac{\cH^{s}_{\infty}(E)}{\frac{\cH^{s}_{\infty}(E)}{(\diam E)^{t}} +\frac{t 2^{s-t}}{s-t} (\diam E)^{s-t}}\\
& \geq\frac{\cH^{s}_{\infty}(E)}{(\diam E)^{s-t}(1+\frac{t 2^{s-t}}{s-t})}
\end{align*}

since $\cH^{s}_{\infty}(E)\leq (\diam E)^{s}$. This finishes the proof.

\end{proof}

\section{The proof of Theorem \ref{t:main}}

Instead of constructing curves like that in the definition of a John domain, it will be more convenient to work with Harnack chains. 

Recall that a {\it Harnack chain (of length $k$)} is a sequence of balls $\{B_{i}\}_{i=1}^{k}$ such that for all $i$,
	\begin{enumerate}
		\item $B_{i}\cap B_{i+1}\neq\emptyset$,
		\item $2B_{i}\subseteq \Omega$, and 
		\item $r_{B_{i}} \sim \dist(B_{i},\d\Omega)$.
	\end{enumerate}
	
\begin{definition}
\label{d:definition}
For $C>0$, a domain $\Omega$ is a {\it $C$-uniform domain} if
\begin{enumerate}
\item it has {\it interior corkscrews}, meaning for every $x\in \d\Omega$ and $0<r<\diam \Omega$, there is a ball of radius $r/C$ contained in $B(x,r)\cap \Omega$, and
\item if $\Lambda(t)=1+\log t$, for all $x,y\in \Omega$, there is a Harnack chain from $x$ to $y$ in $\Omega$ of length $C\Lambda (|x-y|/\min\{\delta_{\Omega}(x),\delta_{\Omega}(y)\})$. 
\end{enumerate}

A domain $\Omega$ is a {\it $C$-chord-arc domain} (or {\it CAD}) if it is $C$-uniform and

\begin{enumerate}
\setcounter{enumi}{2}
\item it has {\it exterior corkscrews}: for every $x\in \d\Omega$ and $r>0$, there is a ball of radius $r/C$ contained in $B(x,r)\backslash \Omega$ and 
\item $\d\Omega$ is {\it Ahlfors $(n-1)$-regular}: for every $x\in \d\Omega$ and $0<r<\diam \d\Omega$, 
\[
C^{-1} r^{n-1}\leq \cH^{n-1}(\d\Omega\cap B(x,r)) \leq Cr^{n-1}. 
\]
\end{enumerate}
We'll say $x$ is the {\it center} of $\Omega$ if
\[
B(x,C^{-1}\diam \Omega)\subseteq \Omega \subseteq B(x,\diam\Omega).
\]

\end{definition}

\begin{remark}
Note that this is slightly different from the definition in \cite{HM14}. There they allow {\it any} function $\Lambda:[1,\infty)\rightarrow [1,\infty)$, but one can show that it is always a constant multiple of $1+\log x$, see \cite{GO79}. Also, to some this definition of unifom domain may not be familiar, but it is equivalent to the  common definition that is in terms of curves, see \cite{AHMNT17}. 
\end{remark}

We now begin the proof of Theorem \ref{t:main}. Let $\Omega$ satisfy the conditions of the theorem and let $0<t<s$. Let $x\in \Omega$. For $y\in \Omega$,  set $\delta_{\Omega}(y)=\dist(y,\Omega^{c})$.

Below, $\alpha$ will denote a multi-index $\alpha=\alpha_{1}...\alpha_{|\alpha|}$ where $|\alpha|$ denotes the length of $\alpha$ and each $\alpha_{j}$ is some integer. We say $\alpha\leq \beta$ if $\alpha$ is an ancestor of $\beta$ (that is, the first $|\alpha|$ terms of $\alpha$ and $\beta$ are the same). We let $x=x_{\emptyset}$ where $\emptyset$ is the empty multi-index and suppose $\delta_{\Omega}(x)=1$. Inductively, we construct a tree of points as follows. 

Let $M>0$ be a large constant we will fix later, and $\ve>0$, a constant we will constantly be adjusting to make smaller but ultimately will only depend on $s,t$, and $n$. Let $k\in \{0,1,...\}$ and suppose we have a point $x_{\alpha}$ with $|\alpha|=k$, and that there is $\xi_{\alpha}\in \d\Omega$ so that
\[
|x_{\alpha}-\xi_{\alpha}|=\delta_{\Omega}(x_{\alpha})=\ve^{|\alpha|}.
\]

Let 
\[
B_{\alpha}=B(x_{\alpha},\ve^{|\alpha|}),\;\;\; E_{\alpha} = B(\xi_{\alpha},\ve^{|\alpha|}/4)\backslash \Omega.
\]
Then
\[
\cH_{\infty}^{s} (E_{\alpha}) \geq \underbrace{c_0 4^{-s}}_{:=c_{1}} \delta_{\Omega}(x_{\alpha})^{s}
=c_{1} \ve^{|\alpha|s}.
\]
By our assumptions, and since $\cH^{s}_{\infty}(E_{\alpha})\leq (\diam E_{\alpha})^{s}$,  if $t'=\frac{t+s}{2}$,
\[
(\diam E_{\alpha})^{t'-s} \cH^{s}(E_{\alpha})
\geq \cH^{s}(E_{\alpha})^{\frac{t'}{s}}
\geq \ps{c_{1} \ve^{|\alpha|s}}^{\frac{t'}{s}}
=c_{1}^{\frac{t'}{s}} \ve^{|\alpha|t'}
\]
Let $\theta>0$ be samll. By Lemma \ref{l:mattila}, we can find $v_{\alpha}$ so that 
\begin{equation}
\label{e:1010}
\av{v_{\alpha}-\frac{\xi_{\alpha}-x_{\alpha}}{|\xi_{\alpha}-x_{\alpha}|}}<\theta
\end{equation}
and if $V_{\alpha}$ is the $(n-1)$-dimensional plane passing through $x_{\alpha}$ perpendicular to $v_{\alpha}$ and $P_{\alpha}$ is the orthogonal projection onto $V_{\alpha}$, then for some constant $c_{2}=c_{2}(s,t,n)$,
\[
\cH^{t'}_{\infty}(P_{\alpha}(E_{\alpha}))\geq c_{2} \ve^{|\alpha|t'} = c_{2} \ve^{k|t'} .
\]
Let $M>1$ and $\{y_{\alpha i}\}_{i\in I_{\alpha}'}\subseteq P_{\alpha}(E_{\alpha})$ be a maximal collection of  points so that  $|y_{\alpha i}-y_{\alpha j}|\geq M\ve^{|\alpha|+1}$ for all $i,j\in I_{\alpha}'$. Let $n_{\alpha}'=|I_{\alpha}'|$. Then the balls $B(y_{\alpha i},M\ve^{|\alpha|+1})$ cover $P_{\alpha}(E_{\alpha})$, and so 
\[
(2M\ve^{|\alpha|+1})^{t'} n_{\alpha}'
=\sum_{i\in I_{\alpha}'} (\diam B(y_{\alpha i},M\ve^{|\alpha|+1}))^{t'}
\geq 
\cH^{t'}_{\infty}(P_{\alpha}(E_{\alpha}))\geq c_{2} \ve^{kt'}.
\]
Recalling $k=|\alpha|$, we pick $\ve>0$ small (depending on $M,t,$ and $s$) so that 
\[
n_{\alpha}' \geq ((2M)^{-t'}c_{2})  \ve^{-t'}> \ve^{-t}.
\]
Now pick $I_{\alpha}\subseteq I_{\alpha'}$ so that, if $n_{\alpha}=|I_{\alpha}|$, then there is $n_{k}$ so that 
\begin{equation}
\label{e:nk}
2\ve^{-t}\geq n_{\alpha}=n_{k}>\ve^{-t}.
\end{equation}
Let 
\[
h_{\alpha i} = \sup\{h>0: B(y_{\alpha i + h v_{\alpha} },\ve^{k+1})\subseteq \Omega\}.
\]
That is, $h_{\alpha i}$ is the farthest one can travel from $y_{\alpha i}$ in the direction $v_{\alpha}$ so that one is at least $\ve^{k+1}$ away from the boundary. These values $h_{\alpha i}$ will be the length of the tentacles we add at this stage.

Let $\theta'>0$ be small. Since $E_{\alpha}\subseteq \frac{3}{2} B_{\alpha}$ and by \eqref{e:1010} for $\theta$ small enough (depending on $\theta'$)
\begin{equation}
\label{e:pea}
y_{\alpha}\in P_{\alpha}(E_{\alpha})\subseteq B(x_{\alpha},(1+\theta')\ve^{|\alpha|}/4)\subseteq \frac{2}{5} B_{\alpha}\subseteq  \Omega,
\end{equation}
we know that 
\[
h_{\alpha i}\leq \frac{1}{2}\diam \frac{3}{2} B_{\alpha} = \frac{3}{2} \ve^{|\alpha|}.
\]
Let 
\[
x_{\alpha i} = y_{\alpha i} + h_{\alpha i}v_{\alpha}
\]
so that (see Figure \ref{f:claws})
\[
\delta_{\Omega}(x_{\alpha i})=\ve^{|\alpha|+1}.
\]

\begin{figure}[!ht]
\includegraphics[width=300pt]{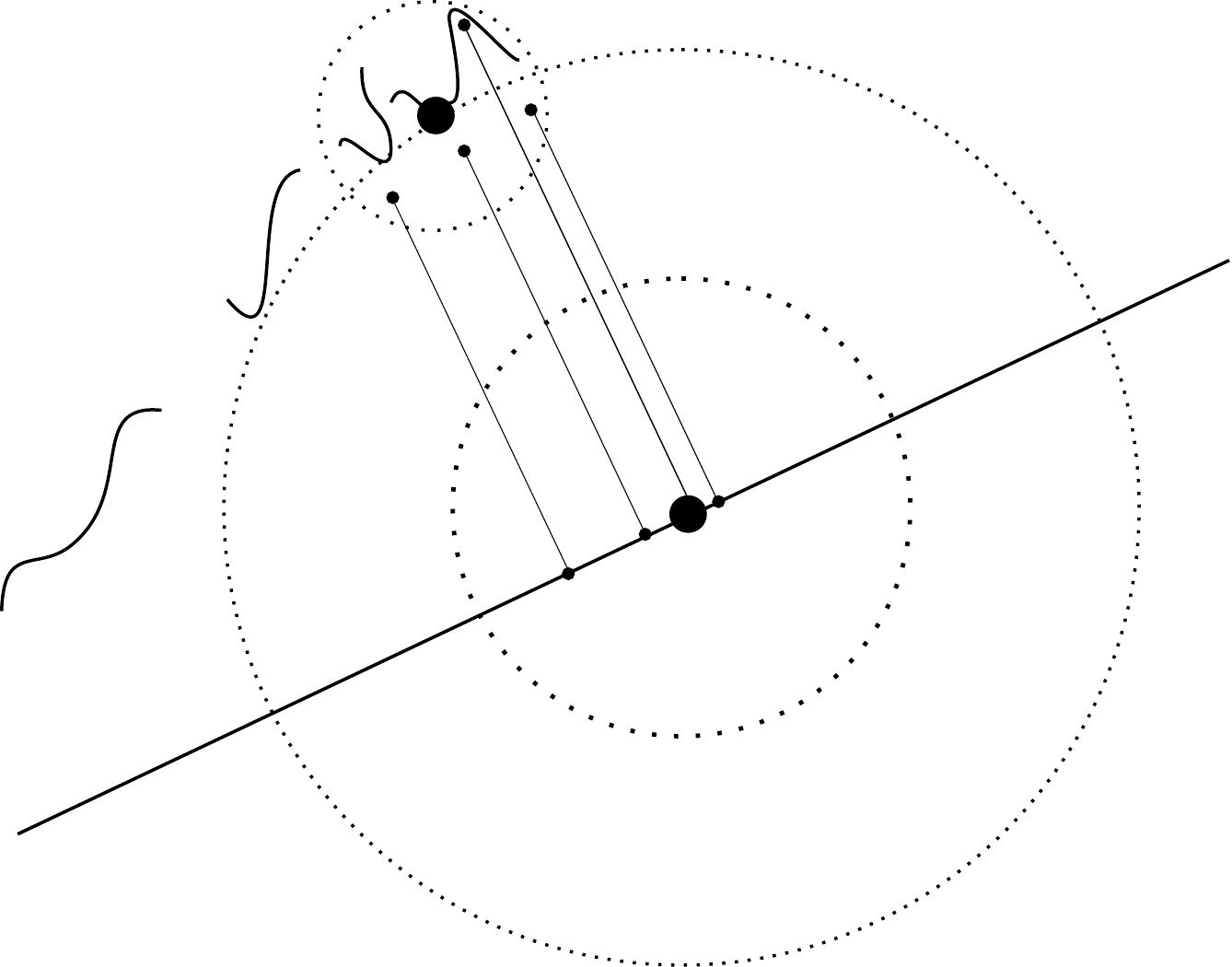}
\begin{picture}(0,0)(300,0)

\put(10,50){$V_{\alpha}$}
\put(160,95){$x_{\alpha}$}
\put(125,80){$y_{\alpha i}$}
\put(45,175){$\Omega^{c}$}
\put(90,220){$\xi_{\alpha}$}
\put(200,155){$\frac{1}{2}B_{\alpha}$}
\put(250,180){$B_{\alpha}$}
\put(75,170){$x_{\alpha i}$}
\put(0,220){$B(\xi_{\alpha},\ve^{|\alpha|}/4)\rightarrow$}
\end{picture}
\caption{Displayed is the point $x_{\alpha}$. In some direction $v_{\alpha}$, the orthogonal projection of $\d\Omega$ has large Hausdorff content, so we can find many points $y_{\alpha i}$ that are in the projection of $\d\Omega$ in $V_{\alpha}\cap \frac{1}{2}B_{\alpha}$ and are $M\ve^{|\alpha|+1}$-separated. We then pick points $x_{\alpha i}$ above these $y_{\alpha i}$ that are distance $\ve^{|\alpha|+1}$ from $\d\Omega$ and so that the segment $[y_{\alpha i},x_{\alpha i}]$ is contained in $\Omega$. These segments are the tentacles that connect $\frac{1}{2}B_{\alpha}$ to the balls $\frac{1}{2}B_{\alpha i}$ that are much closer to the boundary.}
\label{f:claws}
\end{figure}
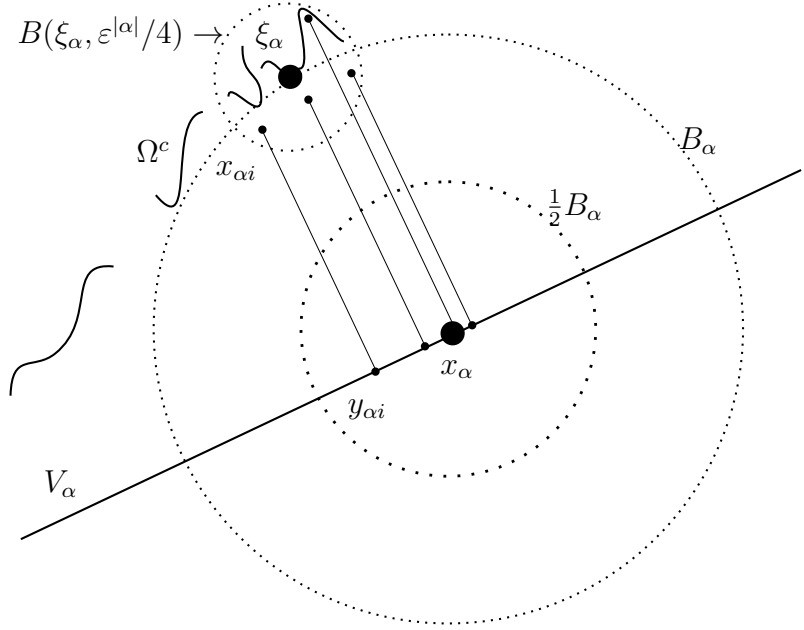

We record a few useful estimates. First, we claim that 
\begin{equation}
\label{e:xxi}
|x_{\alpha i}-\xi_{\alpha}|<\frac{3\ve^{|\alpha|}}{8}.
\end{equation}

To see this, note that if $\xi_{\alpha i}$ is a point closest to $x_{\alpha i}$, then $|\xi_{\alpha i}-x_{\alpha i}|=\ve^{|\alpha i|}=\ve^{|\alpha|+1}$ by construction, so if $\xi_{\alpha i}\in B(\xi_{\alpha},\ve^{|\alpha|}/4)$, then \eqref{e:xxi} is immediate for $\ve>0$ small enough. It is also immediate if $x_{\alpha i}\in B_{\alpha i}$, so assume  $\xi_{\alpha i},x_{\alpha i}\not \in B(\xi_{\alpha},\ve^{|\alpha|}/4)$. Then $x_{\alpha i}\in B_{\alpha}$ since  $[y_{\alpha i},x_{\alpha i}]\subseteq B(\xi_{\alpha},\ve^{|\alpha|}/4)\cup B_{\alpha}$ and $x_{\alpha i}\not\in B(\xi_{\alpha},\ve^{|\alpha|}/4)$. Since $B_{\alpha}\subseteq \Omega$, $\xi_{\alpha i}\not\in B_{\alpha}$. Let $y\in [x_{\alpha i},\xi_{\alpha i}]\cap \d B_{\alpha}$ and $u_{\alpha} =(\xi_{\alpha}-x_{\alpha})/|\xi_{\alpha}-x_{\alpha}|$. Then 
\begin{align*}
|\xi_{\alpha}-y|^2
& =|\xi_{\alpha}-x_{\alpha}|^2+|x_\alpha-y|^2-2(\xi_{\alpha}-x_{\alpha})\cdot(y-x_{\alpha})\\
& =2\ve^{2|\alpha|}\ps{1- \av{u_{\alpha}\cdot \frac{y-x_{\alpha}}{\ve^{\alpha}}}}\\
& \stackrel{\eqref{e:1010}}{ \leq} 2\ve^{2|\alpha|}\theta + 2\ve^{2|\alpha|}\ps{1- \av{v_{\alpha}\cdot \frac{y-x_{\alpha}}{\ve^{\alpha}}}}\\
&   \leq 2\ve^{2|\alpha|}\theta +2\ve^{2|\alpha|}\ps{1- \av{v_{\alpha}\cdot \frac{y-x_{\alpha}}{\ve^{\alpha}}}^2}\\
& =2\ve^{2|\alpha|}\theta +2\ve^{2|\alpha|} \av{P_{\alpha}\ps{ \frac{y-x_{\alpha}}{\ve^{\alpha}}}}^2
=2\ve^{2|\alpha|}\theta +2\av{P_{\alpha}\ps{ {y-x_{\alpha}}}}^2
\end{align*}
Recalling that $P_{\alpha}(x_{\alpha i})=y_{\alpha i}\in \frac{1}{4}B_{\alpha}$, for $\ve,\theta,\theta'>0$ small enough, and since $\sqrt{2}/4<3/8$, we have
\begin{align*}
|\xi_{\alpha}-x_{\alpha i}|
& \leq \ve^{|\alpha|+1}+|\xi_{\alpha}-y|
\leq \ve^{|\alpha|+1}+\sqrt{2}\ve^{|\alpha|}\sqrt{\theta} +  \sqrt{2} |P_{\alpha}(y-x_{\alpha})|\\
& \leq \ve^{|\alpha|+1}+\sqrt{2}\ve^{|\alpha|}\sqrt{\theta} +  \sqrt{2} \ps{|P_{\alpha}(x_{\alpha i}-x_{\alpha})|+\ve^{|\alpha|+1}}\\
&  \stackrel{\eqref{e:pea}}{ \leq}  \ve^{|\alpha|+1}+\sqrt{2}\ve^{|\alpha|}\sqrt{\theta} + \sqrt{2}\ps{(1+\theta')\frac{\ve^{|\alpha|}}{4}+\ve^{|\alpha|+1}}
<\frac{3\ve^{|\alpha|}}{8}
\end{align*}
This proves \eqref{e:xxi}.


%
Thus, for $\ve>0$ small,
\begin{equation}
\label{e:contained}
2B_{\alpha i}
\subseteq B\ps{\xi_{\alpha},\frac{3\ve^{|\alpha|}}{8}+2\ve^{|\alpha|+1}}
\subseteq B\ps{x_{\alpha},\frac{11\ve^{|\alpha|}}{8}+2\ve^{|\alpha|+1}}
\subseteq \frac{4}{3} B_{\alpha}
\end{equation}
where $B_{\alpha i} = B(x_{\alpha i},\ve^{|\alpha i|})$.
Moreover, for $i,j\in I_{\alpha}$ distinct and $M>8$,
\begin{align*}
\dist(2B_{\alpha i},2B_{\alpha j})
& \geq \dist(P_{\alpha}(2B_{\alpha i}),P_{\alpha}(2B_{\alpha i}))\\
& \geq \dist(B(y_{\alpha i},2\ve^{k+1}),B(y_{\alpha j},2\ve^{k+1})) \\
& \geq (M-4)\ve^{k+1} \geq \frac{M}{2} \ve^{k+1}.
\end{align*}
By \eqref{e:contained}, this implies that for any $\alpha$ and $\beta$ of possibly different lengths, if $\gamma$ is the earliest common ancestor of $\alpha$ and $\beta$ and $\gamma\neq\alpha,\beta$, then
\begin{equation}
\label{e:far}
\dist(2B_{\alpha},2B_{\beta})
\geq \frac{M}{2} \ve^{|\gamma|+1}.
\end{equation}
In particular, 
\begin{equation}
\label{e:far-same-gen}
\dist(2B_{\alpha},2B_{\beta})\geq \frac{M}{2}\ve^{k}\;\; \mbox{ if }\;\; |\alpha|=|\beta|=k \mbox{ and }\alpha\neq\beta.
\end{equation}

We will also need the following estimate bounding how close a ball is from the center of its parent ball: for $\ve>0$ small,
\begin{align}
\dist\ps{\frac{1}{2} B_{\alpha},2B_{\alpha i}}
& \geq |x_{\alpha} - x_{\alpha i}| - \frac{\ve^{|\alpha|}}{2} -2\ve^{|\alpha|+1}\notag \\
& \stackrel{\eqref{e:xxi}}{\geq}  |x_{\alpha }-\xi_{\alpha }| - \frac{3\ve^{|\alpha|}}{8}- \frac{\ve^{|\alpha|}}{2}(1+4\ve) \notag \\
& =\ve^{|\alpha|} - \frac{3\ve^{|\alpha|}}{8}- \frac{\ve^{|\alpha|}}{2}(1+4\ve)
>\frac{\ve^{|\alpha|}}{9}.
\label{e:spaced}
\end{align}

\begin{lemma}
\label{l:frostmann}
Let $E\subseteq \d\Omega$ be the set of points $z$ for which there is a sequence of multi-indices $\alpha_{k}$ with $|\alpha_{k}|=k$ and $x_{\alpha_{k}}\rightarrow z$. Then
\[
\cH^{t}_{\infty}(E)\gec 1. 
\]
\end{lemma}

\begin{proof}
Let us define a sequence of probability measures $\mu_{k}$ as follows. We first let $\mu_{0}$ be a measure so that $\mu_{0}(2B_{\emptyset})=1$. Inductively, and using \eqref{e:contained} we let $\mu_{k}$ be a measure so that 
\[
\mu_{k}(2B_{\alpha i}) = \frac{\mu_{k-1}(2B_{\alpha})}{n_{\alpha}}<\mu_{k-1}(2B_{\alpha})\ve^{t} \;\; \mbox{ for all }i\in I_{\alpha}.
\]
By passing to a weak limit, we obtain a measure $\mu$ supported on $E$ so that if $\alpha'$ denotes the string $\alpha$ minus its last term, then
\[
\mu(2B_{\alpha}) = \frac{\mu(2B_{\alpha'})}{n_{\alpha'}}\stackrel{\eqref{e:nk}}{<}\mu(2B_{\alpha'})\ve^{t}<\cdots < \ve^{|\alpha|t} \;\; \mbox{ for all }i\in I_{\alpha}.
\]
In particular, if $B$ is any ball intersecting $E$ with $\diam B<1$, let $k$ be such that $\frac{M}{4}\ve^{k}>\diam B\geq \frac{M}{4}\ve^{k+1}$. Then there is at most one $2B_{\alpha}$ with $|\alpha|=k$ intersecting $B$; otherwise, if $\beta$ was another such multi-index, then
\[
\frac{M}{2} \ve^{k}
\stackrel{\eqref{e:far-same-gen}}{\leq} \dist(2B_{\alpha},2B_{\beta})
\leq \diam B<\frac{M}{4} \ve^{k}
\]
which is a contradiction. Thus,
\[
\mu(B)\leq \mu(2B_{\alpha})<\ve^{tk}\lec (\diam B)^{t}.
\]
If $\diam B\geq 1$, then 
\[
\mu(B)\leq \mu(\bR^{n})=1\leq (\diam B)^{t}.
\]
Thus, $\mu$ is a $t$-Frostmann measure, so $\cH^{t}_{\infty}(E)\gec_{n} \mu(E)=1$. 
\end{proof}

Fix an integer $N$ and let $W$ denote the Whitney cubes for $\Omega$, which we define to be the set of maximal dyadic cubes $Q$ so that 
\[
NQ\subseteq \Omega. 
\]
Let $\lambda>1$. For $\alpha$ a multi-index, let
\[
\cC(\alpha) = \left\{Q\in W:\;\; Q\cap \ps{\frac{1}{2} B_{\alpha} \cup\bigcup_{i\in I_{\alpha}}[y_{\alpha i},x_{\alpha i}]}\neq\emptyset \right\}\]
where $[x,y]$ denotes the closed line segment between $x$ and $y$, and 
\[
\Omega_{\alpha}=\bigcup_{Q\in \cC_{\alpha}} \lambda Q. 
\]
We now pick $N$ large enough so that by \eqref{e:contained},
\begin{equation}
\label{e:ominb}
\Omega_{\alpha}\subseteq \frac{5}{4}B_{\alpha}
\end{equation}
and so that
\begin{equation}
\label{e:Q1/2B}
\lambda Q\subseteq \frac{3}{4} B_{\alpha}\mbox{ for all $Q\in W$ so that $\lambda Q\cap \frac{1}{2}B_{\alpha}\neq\emptyset$}.
\end{equation}

Note that all the Whitney cubes $Q\in \cC_{\alpha}$, have comparable sizes (depending on $\ve$), there are boundedly many such cubes. Since $\frac{1}{2} B_{\alpha} \cup\bigcup_{i\in I_{\alpha}}[y_{\alpha i},x_{\alpha i}]$ is connected, so is $\Omega_{\alpha}$. Because of this, it is not hard to show that, for $\lambda$ close enough to $1$, $\Omega_{\alpha}$ is a CAD with constants depending only on $\ve$, $\lambda$, and $n$. Here, $\lambda$ is a universal constant depending on $n$ and is now fixed.

Also, if 
\[
\Omega(\alpha)=\bigcup_{\beta\geq \alpha} \Omega_{\beta},
\]
then
\begin{equation}
\label{e:o,imb2}
\Omega(\alpha)\subseteq 2B_{\alpha}. 
\end{equation}

Let $\cC=\bigcup_{\alpha} \cC(\alpha)$ and 
\[
\Omega(x) = \bigcup_{\alpha}\Omega_{\alpha} = \bigcup_{Q\in \cC} \lambda Q.
\]

Note that by construction. $E\subseteq \d\Omega(x)\cap \d\Omega$.

\begin{lemma}
The domain $\Omega(x)\subseteq \Omega$ is $C$-uniform with $C$ depending on $\ve$ and $n$ so that $\d\Omega\cap \d\Omega(x)=E$.
\end{lemma}

\begin{proof}
The last part of the lemma follows from the discussion that preceded it, so we just need to verify that $\Omega(x)$ is uniform. By construction, $\Omega(x)$ satisfies the interior corkscrew property, and so we just need to bound the length of Harnack chains. As in the proofs of  \cite[A.1]{HM14}, since the $\Omega_{\alpha}$ are themselves uniform, it suffices to show that we may connect each $x_{\alpha}$ and $x_{\beta}$ by Harnack chains of the correct length.

Let $\gamma$ be the earliest common ancestor of $\alpha$ and $\beta$. Let $k_{\alpha}=|\alpha|-|\gamma|$, $k_{\beta}=|\beta|-|\gamma|$, and let $\alpha_{j}$ be the ancestor of $\alpha$ so that $|\alpha_{j}|=|\gamma|+j$. Note that since $\delta_{\Omega}(x_{\alpha_{j}})=\ve^{|\gamma|+j} = \ve \delta_{\Omega}(x_{\alpha_{j+1}})$ and by construction of $\Omega_{\alpha}$, we know $\delta_{\Omega_{\alpha_{j}}}(x_{\alpha_{j}})\sim \delta_{\Omega_{\alpha_{j}}}(x_{\alpha_{j+1}})\sim \ve^{|\gamma|+j} $, and since both $x_{\alpha_{j}}$ and $x_{\alpha_{j+1}}$ are contained in  $2B_{\alpha_{j}}$, we know $|x_{\alpha_{j}}-x_{\alpha_{j+1}}|\leq  2\ve^{|\gamma|+j}$. Thus, since  $\Omega_{\alpha_j}$ is a CAD, 
there is a Harnack chain in $\Omega_{\alpha_{j}}$of uniformly bounded length (depending on the uniformity constants for $\Omega_{\alpha_{j}}$) from $x_{\alpha_{j}}$ to $x_{\alpha_{j+1}}$. The union of the Harnack chains for each $j$ gives a Harnack chain from $x_{\gamma}$ to $x_{\alpha}$ of length comparable to $k_{\alpha}$. We can find another Harnack chain from $x_{\gamma}$ to $x_{\beta}$ of length $k_{\beta}$. Now we just need to estimate the length of the total chain. By \eqref{e:ominb} and \eqref{e:far}, 
\[
|x_{\alpha}-x_{\beta}|
\geq \frac{M}{2} \ve^{|\gamma|+1}.\]
Also, by definition of $\Omega_{\alpha}$ and $\Omega(x)$, we have 
\[
\delta_{\Omega(x)}(x_{\alpha})\sim \delta_{\Omega}(x_{\alpha})=\ve^{|\alpha|}.
\]
Thus, the length of the Harnack chain is at most a constant times
\begin{align*}
k_{\alpha}+k_{\beta}
& \leq 2\max\{k_{\alpha},k_{\beta}\}
\lec_{\ve} 1+ \log \frac{\ve^{|\gamma|}}{\min \{\ve^{|\alpha|},\ve^{|\beta|}\}}\\
& \lec 1+ \log \frac{|x_{\alpha}-x_{\beta}|}{\min\{\delta_{\Omega(x)}(x_{\alpha}),\delta_{\Omega(x)}(x_{\beta})\}}.
\end{align*}

Thus, the conditions for being uniform hold.
\end{proof}

\begin{lemma}
$\Omega(x)$ has exterior corkscrews. 
\end{lemma}

\begin{proof}
Let $z\in \d\Omega(x)$ and $r>0$. We divide into some cases:

\noindent {\bf Case 1:} If $r\geq 2 \diam \Omega(x)$, then we can clearly find a corkscrew ball in $B(z,r)\backslash \Omega(x)$ of radius $r/4$. 

\noindent {\bf Case 2:} Assume $r<2\diam \Omega$. Let $C>0$, we will decide its value soon. 

\noindent {\bf Case 2a:} Suppose $0<r<C\delta_{\Omega}(z)$, then $z\in \d \lambda Q$ for some Whitney cube $Q\in \cC$. Note that for $\rho>0$ small enough (depending on $N$ and $n$), $\d\Omega(x)\cap B(z,\rho \ell(Q))\backslash\Omega(x)$ is isometric to $B(0,\rho \ell(Q))\backslash \{y:y_{i}\geq 0\}_{i\in S}$ for some subset  $S\subseteq \{1,...,n\}$, hence we can find a ball of radius $\rho \ell(Q)/4\subseteq B(z,\rho \ell(Q))\backslash \Omega(x)$. 

By the properties of Whitney cubes, 
\[
r<C\delta_{\Omega}(z)\sim C\ell(Q)\lec_{\rho} C\rho \ell(Q)/4,\]
This means the ball is a corkscrew ball for $B(z,r)$ with respect to $\Omega(x)$. 

\noindent {\bf Case 2b:} Now suppose $r\geq C\delta_{\Omega}(z)$. Note that if $Q\in \cC$, then $Q\in \cC_{\alpha}$ for some $\beta$, and by \eqref{e:o,imb2}, if $z\in \d \lambda Q$,
\[
\dist(z,E)
\leq \diam 2B_{\beta}. \]

Also note that $Q$ has side length comparable to every other cube in $\cC(\beta)$ (since $\Omega_{\beta}$ is a finite connected union of dilated Whitney cubes), so in particular, if $R\in \cC(\beta)$ is such that $x_{\beta}\in \lambda R$,
\[
\delta_{\Omega}(z)\sim \ell(Q) \sim \ell(R)\sim \delta_{\Omega}(x_{\beta})=  |x_{\beta}-\xi_{\beta}|\sim \diam B_{\beta} .
\]
Combining the above inequalities, we get 
\[
\dist(z,E)\lec \delta_{\Omega}(z)\leq r/C,
\]
 so for $C$ large enough,
 \[
 \dist(z,E)<r/2.
 \]
 Hence, we can pick $w\in E\cap B(z,r/2)$. Note there is a sequence $\alpha_{k}$ so that $|\alpha_{k}|=k$, $\alpha_{k}\leq \alpha_{k+1}$, and $x_{\alpha_{k}}\rightarrow w$. Let $\alpha=\alpha_{k}$ be so that
 \[
 \diam \Omega_{\alpha} = \max\{ \diam \Omega_{\alpha_{k}}:  \Omega_{\alpha_{k}}\subseteq B(w,r/4)\}.
 \]
Since $\diam \Omega(\alpha)\sim  \diam B_{\alpha}= 2\ve^{|\alpha|}$ and $r<2\diam \Omega(x)$, it follows that $\diam \Omega(\alpha) \sim_{\ve} r$. 
 
 Note that if $\alpha'=\alpha_{k-1}$ is the parent of $\alpha$, then by \eqref{e:spaced} and \eqref{e:far}, and because the $I_{\alpha _{j}}$ are mutually spaced apart by distance at least $M\ve^{|\alpha|}$, we have for $\ve>0$ small enough and $M$ large enough (and here we fix $M$)
 
 \[
 3B_{\alpha}\cap \Omega(x)
 =\Omega(\alpha)\cup \bigcup\{\lambda Q: Q\in W, \;\; Q\cap I_{\alpha'}\neq\emptyset\}.
 \]
 Hence, for $\rho>0$ and $N$ large enough depending on $\rho$, and as $\Omega(\alpha)\subseteq 2B_{\alpha}$,
 \[
\sup\{\dist(y, 2B_{\alpha}\cup I_{\alpha'}): y\in 3B_{\alpha}\cap \Omega(x)\}
<\rho \ve^{|\alpha|}
.\]
For $\rho$ small enough, this means there is $B^{\alpha}\subseteq 3B_{\alpha}\backslash \Omega(x)$ of radius $\ve^{|\alpha|}/4\sim_{\ve} r$, so this in turn will be an exterior corkscrew for $\Omega(x)$ in $B(w,r/2)$.

\end{proof}

\begin{lemma}
$\d\Omega(x)$ is Ahlfors $(n-1)$-regular.
\end{lemma}

\begin{proof}
Let $z\in \d\Omega(x)$ and $0<r<\diam \Omega(x)$. The interior and exterior corkscrew conditions imply lower regularity; this is standard, but it's short enough to include here: We know there are balls $B(x_{1},c r)\subseteq  B(z,r)\cap \Omega(x)$ and $B(x_{2},c r)\subseteq B(z,r)\backslash \Omega(x) $ with $c=c(\ve,n)$. If $U$ is the $(n-1)$-dimensional plane perpendicular to $x_1-x_2$ passing through $0$ and $P$ is the orthogonal projection onto $U$, then
\begin{multline*}
\cH^{n-1}(B(z,r)\cap \d\Omega(x))
 \geq \cH^{n-1}(P(B(z,r)\cap \d\Omega(x)))\\
 \geq \cH^{n-1}(P(B(x_{1},cr))\cap U)=\cH^{n-1}(B(P(x_{1}),cr)\cap U)
 \gec_{c,d} r^{n-1}.
\end{multline*}

Now we prove upper regularity.  Again, as in the proof of Lemma \ref{l:frostmann}  if $k$ is such that $\frac{M}{4}\ve^{k}>\diam B\geq \frac{M}{4}\ve^{k+1}$, then there is at most one $2B_{\alpha}$ with $|\alpha|=k$ intersecting $B$. Hence, $B$ touches only $\cnj{\Omega(\alpha)}$.

Each $\d\Omega_{\alpha}$ is already Ahlfors regular and $\cH^{n-1}(\d\Omega_{\alpha})\lec \ve^{|\alpha|(n-1)} $. By \eqref{e:nk} there are $n_{|\alpha'|}\cdots n_{k-1}$ many descendants $\beta$ of $\alpha'$ with $|\beta|=k$, and 
\[
n_{|\alpha'|}\cdots n_{k-1}
\stackrel{\eqref{e:nk}}{<}(2\ve^{-t})^{k-1-|\alpha'|+1}= (2\ve^{-t})^{k-|\alpha|+1}.
\]
Thus, for $\ve>0$ small enough depending on $n$ and $t$,
\begin{align*}
\cH^{n-1}(B(z,r)\cap \d\Omega(x))
& = \cH^{n-1}(B(z,r)\cap \d\Omega(\alpha'))
 \leq \sum_{\beta\geq \alpha'} \cH^{n-1}(\d\Omega_{\beta})\\
& \lec \sum_{k\geq |\alpha'|}\ve^{k(n-1)} \cdot (2\ve^{-t})^{k-|\alpha|+1}\\
& = 2^{-|\alpha|+1} \ve^{t(|\alpha|-1)} \sum_{k\geq |\alpha'|} \ve^{k(n-1-t)}2^{k}\\
& \lec_{\ve} 2^{-|\alpha|} \ve^{t|\alpha|} \cdot \ve^{|\alpha'|(n-1-t)}2^{|\alpha'|}
\lec \ve^{|\alpha|(n-1)}
\lec r^{n-1}.
\end{align*}

\end{proof}

The combination of the previous four lemmas prove Theorem \ref{t:main}.
%
%
%
%
%
%
%
%
%
%
%
%

\bibliographystyle{alpha}

\begin{thebibliography}{AHM{\etalchar{+}}17}

\bibitem[AHM{\etalchar{+}}17]{AHMNT17}
J.~Azzam, S.~Hofmann, J.M. Martell, K.~Nystr\"om, and T.~Toro.
\newblock A new characterization of chord-arc domains.
\newblock {\em J. Eur. Math. Soc. (JEMS)}, 19(4):967--981, 2017.

\bibitem[BJ90]{BJ90}
C.~J. Bishop and P.~W. Jones.
\newblock Harmonic measure and arclength.
\newblock {\em Ann. of Math. (2)}, 132(3):511--547, 1990.

\bibitem[Fal86]{Falconer}
K.~J. Falconer.
\newblock {\em The geometry of fractal sets}, volume~85 of {\em Cambridge
  Tracts in Mathematics}.
\newblock Cambridge University Press, Cambridge, 1986.

\bibitem[GO79]{GO79}
F.~W. Gehring and B.~G. Osgood.
\newblock Uniform domains and the quasihyperbolic metric.
\newblock {\em J. Analyse Math.}, 36:50--74 (1980), 1979.

\bibitem[HM14]{HM14}
S.~Hofmann and J.~M. Martell.
\newblock Uniform rectifiability and harmonic measure {I}: {U}niform
  rectifiability implies {P}oisson kernels in {$L^p$}.
\newblock {\em Ann. Sci. \'Ec. Norm. Sup\'er. (4)}, 47(3):577--654, 2014.

\bibitem[HM18]{HM18}
S. Hofmann and J.M. Martell.
\newblock Harmonic measure and quantitative connectivity: geometric characterization of the $L^p$-solvability of the Dirichlet problem. Part I.
\newblock {\em arXiv preprint arXiv:1712.03696}, 2018.

\bibitem[ILTV14]{ILTV14}
L. Ihnatsyeva, J. Lehrb\"ack, H. Tuominen, and A.V.
  V\"ah\"akangas.
\newblock Fractional {H}ardy inequalities and visibility of the boundary.
\newblock {\em Studia Math.}, 224(1):47--80, 2014.

\bibitem[KL09]{KL09}
P. Koskela and J. Lehrb\"ack.
\newblock Weighted pointwise {H}ardy inequalities.
\newblock {\em J. Lond. Math. Soc. (2)}, 79(3):757--779, 2009.

\bibitem[KNN18]{KNN18}
P.~Koskela, D.~Nandi, and A.~Nicolau.
\newblock Accessible parts of boundary for simply connected domains.
\newblock {\em Proc. Amer. Math. Soc.}, 146(8):3403--3412, 2018.

\bibitem[Leh09]{Leh09}
J. Lehrb\"ack.
\newblock Necessary conditions for weighted pointwise {H}ardy inequalities.
\newblock {\em Ann. Acad. Sci. Fenn. Math.}, 34(2):437--446, 2009.

\bibitem[Leh14]{Leh14}
J. Lehrb\"ack.
\newblock Weighted {H}ardy inequalities beyond {L}ipschitz domains.
\newblock {\em Proc. Amer. Math. Soc.}, 142(5):1705--1715, 2014.

\bibitem[Mat95]{Mattila}
P.~Mattila.
\newblock {\em Geometry of sets and measures in {E}uclidean spaces}, volume~44
  of {\em Cambridge Studies in Advanced Mathematics}.
\newblock Cambridge University Press, Cambridge, 1995.
\newblock Fractals and rectifiability.

\end{thebibliography}

\newcommand{\etalchar}[1]{$^{#1}$}
\def\cprime{$'$}

\end{document}